\documentclass[11pt]{article}

\usepackage{authblk}
\usepackage{amssymb,latexsym,amsmath,amsthm,enumitem}
\usepackage{stmaryrd}

\usepackage{amsfonts}

\usepackage{fullpage}

\usepackage[latin1]{inputenc}

\theoremstyle{definition}
\newtheorem{theorem}{Theorem}
\newtheorem{corollary}[theorem]{Corollary}
\newtheorem{proposition}[theorem]{Proposition}
\newtheorem{lemma}[theorem]{Lemma}
\newtheorem{definition}[theorem]{Definition}
\newtheorem{example}[theorem]{Example}

\newtheorem{notation}[theorem]{Notation}
\newtheorem{remark}[theorem]{Remark}
\usepackage{multirow}
\usepackage{graphicx}
\usepackage{tikz}
\usetikzlibrary{matrix,decorations.pathreplacing}
\usepackage{multirow}

\usepackage[all,cmtip]{xy}
\usepackage{hyperref}

\newcommand{\numberset}{\mathbb}
\newcommand{\fq}{{\mathbb{ F}}_q}
\newcommand{\N}{\numberset{N}}

\newcommand{\Q}{\numberset{Q}}

\newcommand{\C}{\numberset{C}}

\newcommand{\F}{\numberset{F}}

\newcommand{\mC}{\mathcal{C}}
\newcommand{\mE}{\mathcal{E}}

\newcommand{\mA}{\mathcal{A}}

\newcommand{\cC}{\mathcal{C}}

\newcommand{\mD}{\mathcal{D}}

\newcommand\qbin[3]{\left[\begin{matrix} #1 \\ #2 \end{matrix} \right]_{#3}}
\newcommand{\tr}{{\rm Tr}\,}
\newcommand{\rk}{{\rm rk}}
\def\ff{{\mathbb F}}

\usepackage{accents}
\usepackage{enumitem}

\usepackage[margin=2.1cm]{geometry}

\newlength{\dhatheight}

\usepackage{wrapfig}

\usepackage{mathptmx}
\newcommand{\ma}{\F_q^{k \times m}}

\usepackage{titling}

\begin{document}

\title{\textbf{Covering Radius of Matrix Codes Endowed with the Rank
Metric}\vspace{2em}}

%
%
%
%
%

\author{
  Eimear Byrne \\
  \textit{School of Mathematics and Statistics} \\ 
  \textit{University College Dublin,
Ireland}
\\ \texttt{ebyrne@ucd.ie}
  \and
   Alberto Ravagnani \\
  \textit{Institut de Math\'{e}matiques} \\ 
  \textit{Universit\'{e} de Neuch\^{a}tel,
Switzerland}
\\ \texttt{alberto.ravagnani@unine.ch}
}

\date{}

\maketitle

\begin{abstract}
In this paper we study properties and invariants of matrix codes endowed with
the rank metric,
and relate them to the covering radius. We introduce new tools for the 
analysis of rank-metric codes, such as puncturing and shortening constructions.
We give upper bounds on the covering radius of a code by applying different
combinatorial methods.
We apply the various bounds to the classes of maximal rank distance and quasi
maximal rank distance codes.
 
\end{abstract}

\vspace{2em}

\section*{Introduction}

Rank-metric codes have featured prominently in the literature on algebraic codes
in recent years and 
especially since their applications to error-correction in networks were
understood. Such codes are subsets of the matrix ring $\fq^{k \times m}$ endowed
with the rank distance function, 
which measures the $\fq$-rank of the difference of a pair of matrices.
An analogue of the Singleton bound was given in \cite{D78}. If a code meets 
this bound it is referred to as a {\em maximum rank distance} (MRD) code.
It is known that there exist codes meeting this bound for all values 
of $q,k,m,d$ \cite{D78,gabid,john}.
For this reason the {\em main coding problem} for rank metric codes, unlike 
the same problem for the Hamming metric, is closed: for any $q,k,m,d$ the
optimal size 
of a rank-metric code in $\fq^{k \times m}$ of minimum rank distance $d$ is
known. There are very few classes of
rank-metric codes known, due in part to the Delsarte-Gabidulin family and its
generalizations \cite{D78,gabid,john}, which are optimal and can be efficiently
decoded
\cite{gabid,L06,WAS13}.

The {\em covering radius} of a code is a fundamental parameter. It measures 
the maximum weight of any correctable error in the ambient space. 
It also characterizes the {\em maximality} property of a code, that is,
whether or not the code is contained in another of the same minimum distance.
The covering
radius of a code measures the least integer $r$ such that every element of the
ambient space is within distance $r$ of some codeword. This quantity is
generally much harder to compute than the minimum distance of a code. There are
numerous papers and books on this topic for classical codes with respect to the
Hamming
distance (see \cite{BGP15,coveringcodes,CKMS85,CLLM97,GMR} and the references
therein), 
but relatively little attention has been paid to it for
rank-metric codes \cite{gadu1,gadu2}.    

In this paper we describe properties of rank-metric codes and relate these to
the 
covering radius. We define new parameters and give tools for the analysis of
such codes. In particular, we introduce new definitions for
the puncturing and the shortening of a general rank-metric code.
In many instances our tools are applied to establish
new bounds on the rank-metric covering radius. Some of the derived bounds, such
as the dual distance and external distance bounds, are analogues of known bounds
for the Hamming distance. Others, such as the initial set bound, are unique to
matrix codes. We apply our results to the classes of maximal rank distance and
quasi maximal rank distance codes. 

In Section \ref{sec1} we consider the property of {\em maximality}. A code is
maximal if 
it is not contained in another code of the same minimum distance. We introduce
a
new parameter, called the {\em maximality degree} of a code, and show that it is
determined by minimum distance and covering radius of a code. These results are
independent of the metric. In Section \ref{sec2} we define shortened and
punctured codes rank metric codes and describe their properties. We give a
duality result relating a shortened and punctured code. In Section \ref{sec3}
we investigate translates of a code. We show that the weight enumerator of a
coset of a linear code of rank weight is completely determined by the weights of
first $n-d^\perp$ cosets, and establish this using M\"{o}bius inversion on the
lattice of subspaces of $\fq^k$. This is then applied to get the rank-metric
analogue of the {\em dual distance bound}. We also give the rank-metric
generalization of the {\em external distance bound}, which holds also for
non-linear codes. In Section \ref{sec4} we 
introduce the concept of the {\em initial set} of a matrix code and use this 
to derive a bound on the covering radius of a code. In Section \ref{sec5} we
apply previously derived bounds to maximum rank distance and quasi maximum rank
distance codes.

\section{Preliminaries} \label{sec0}

Throughout this paper, $q$ is a fixed prime power, 
$\F_q$ is the finite field with $q$ elements, and 
$k,m$ are positive integers. We assume $k \le m$ without 
loss of generality, and denote by 
$\F_q^{k \times m}$ the  space of 
$k \times m$ matrices over $\F_q$.
For any positive integer $n$ we set $[n]:=\{ i \in \N
:
1 \le i \le n\}$.

\begin{definition}
The \textbf{rank distance} between matrices 
$M,N \in \ma$ is $d(M,N):=\mbox{rk}(M-N)$.
 A \textbf{rank-metric code} is a non-empty subset
 $\mC \subseteq \ma$. When 
 $|\mC| \ge 2$, the \textbf{minimum rank 
distance} of 
 $\mC$ is the integer defined by 
 $d(\mC):= \min\{d(M,N) : M,N \in \mC, \ M \neq N\}$.
 The \textbf{weight} and \textbf{distance}
 \textbf{distribution} of a code $\mC \subseteq \ma$ are the integer vectors 
 $W(\mC)=(W_i(\mC) : 0 \le i \le k)$ and $B(\mC)=(B_i(\mC) : 0 \le i \le k)$,
 where, for all $i \in \{0,...,k\}$, 
 $$W_i(\mC):=|\{M \in \mC : \mbox{rk}(M)=i\}|, \ \ \ \ \ \ \ 
 B_i(\mC):= 1/|\mC| \cdot |\{(M,N) \in \mC \times \mC : d(M,N)=i\}|.$$
  \end{definition}

It is easy to see that $d$ defines a distance function on
$\ma$. 

\begin{definition}
A code $\mC \subseteq \ma$ is \textbf{linear} if it is an 
 $\F_q$-subspace of $\ma$. If this is the case, then 
 the \textbf{dual code} of $\mC$ is the linear code
 $\mC^\perp:= \{N \in \ma : \mbox{Tr}(MN^t)=0 \mbox{ for all }
 M \in \mC\} \subseteq \ma$.
 \end{definition}

 If $\mC \subseteq \ma$ is a linear code then one can easily check that 
 $d(\mC)= \min\{\mbox{rk}(M) : M \in \mC, \ M \neq 0\}$ and  
 $W_i(\mC)=B_i(\mC)$ for all $i \in \{0,...,k\}$.
 Moreover, since the map $(M,N) \mapsto \mbox{Tr}(MN^t)$ defines an inner
product
  on the space $\ma$, 
 we have $\dim(\mC^\perp)=km-\dim(\mC)$ and $\mC^{\perp\perp}=\mC$.

 \begin{definition}
  The \textbf{covering radius} of a code $\mC \subseteq \F_q^{k \times m}$ is the integer
$$\rho(\mC):= \min\{i : \mbox{for all 
$X \in \ma$ there exists $M \in \mC$ with $d(X,M) \le i$}\}$$
 \end{definition}

 In words, the covering radius 
 of a code $\mC$ is the maximum distance of $\mC$ to any matrix in the
ambient space, or the minimum value $r$ such that the 
 	union of the spheres of radius $r$ about each codeword cover the
 ambient
space.
 The following result summarizes some simple properties of this 
 invariant. 
 These facts are known from studies of the Hamming distance
covering radius and, being actually independent of the metric used, hold also in
the rank metric case.
 For a comprehensive treatment of the covering problem for Hamming metric codes,
see \cite{CKMS85,CLLM97}.
 
 \begin{lemma} \label{rd}
  Let $\mC \subseteq \ma$ be a code. The following hold.
  \begin{enumerate}
   \item $0 \le \rho(\mC) \le k$. Moreover, $\rho(\mC)=0$
if and only if $\mC=\ma$.
\item If $\mD \subseteq \ma$ is a code with 
$\mC \subseteq \mD$, then $\rho(\mC) \ge \rho(\mD)$.
\item \label{p3} If $\mD \subseteq \ma$ is a code with 
$\mC \subsetneq \mD$, then $\rho(\mC) \ge d(\mD)$.
\item \label{p4} $d(\mC)-1 < 2\rho(\mC)$, {if $|\mC| \ge 2$ and $\mC
\subsetneq \ma$.}
  \end{enumerate}
 \end{lemma}
 \begin{proof}
  To see that \ref{p3} holds, let $N \in \mD \setminus \mC$. By definition of
covering radius,
  there exists a matrix  
  $M \in \mC$ with $d(M,N) \le \rho(\mC)$. Thus $d(\mD) \le d(M,N) \le \rho(\mC)$.
  
  To see \ref{p4}, observe that the packing radius $\lfloor (d(\mC)-1)/2
\rfloor$ of $\cC$ cannot exceed the covering radius,
  and that equality occurs if and only if $\cC$ is perfect, in which case we have 
 $\lfloor (d(\mC)-1)/2 \rfloor = \rho(\mC)$. However there are no perfect
codes for the rank metric \cite{C96}.   
  \end{proof}

\section{Maximality} \label{sec1}

In this short section we investigate some connections between the covering
radius 
of a rank-metric code and the property of maximality.
Recall that a code  $\mC \subseteq \F_q^{k \times m}$  is  \textbf{maximal} if
$|\mC|=1$ or
  $|\mC| \ge 2$ and 
 there is no code $\mD \subseteq \ma$ with 
 $\mD \varsupsetneq \mC$ and $d(\mD)=d(\mC)$.
In particular, $\F_q^{k \times m}$ is maximal.

\begin{proposition}[see e.g. \cite{CKMS85}] \label{max}
 A code $\mC  \subseteq \F_q^{k \times m}$ with 
 $|\mC| \ge 2$ is 
 maximal if and only if $\rho(\mC)
\le d(\mC)-1$.
\end{proposition}

\begin{proof}
If $\mC$ is not maximal, then there exists 
$\mC \subsetneq \mD$ with 
$d(\mD)=d(\mC)$. Lemma \ref{rd} implies
$\rho(\mC) \ge d(\mC)=d(\mD)$, i.e.,
$\rho(\mC)>d(\mC)-1$. This shows $(\Leftarrow)$.
Let us prove $(\Rightarrow)$.
 If $\mC= \F_q^{k \times m}$ then the result is trivial. Therefore we assume
$\mC\varsubsetneq \F_q^{k \times m}$ and $\rho(\mC) \ge d(\mC)$ by contradiction.
By the definition of covering radius there exists $X \in \F_q^{k \times m} \setminus \mC$ such that 
$d(M,X) \ge \rho(\mC)$ for all matrices $M \in \mC$.
Then the code $\mD:=\mC \cup \{X \}$ strictly contains $\mC$ and has $d(\mD)=d(\mC)$.
\end{proof}

We now propose a new natural parameter that measures the maximality of a code,
and show how it relates
to the covering radius.

\begin{definition}
 The \textbf{maximality degree} of a code $\mC \subseteq \ma$ with 
 $|\mC| \ge 2$
is the integer defined by
$$\mu(\mC):= \left\{ 
\begin{array}{cl}
 \min \{ d(\mC)-d(\mD) : \mD \subseteq \ma \mbox{ is a code with } \mD \varsupsetneq \mC\}
 & \mbox{ if $\mC \subsetneq \ma$,} \\
1  & \mbox{ if $\mC=\ma$.}\end{array} \right.\ $$
\end{definition}

 The maximality degree of a code $\mC \subseteq \ma$ with $|\mC| \ge 2$ satisfies 
 $0 \le \mu(\mC) \le d(\mC)-1$. Moreover, it is easy to see that 
$\mu(\mC)>0$ if and only if $\mC$ is maximal. Notice that 
$\mu(\mC)$ can be interpreted as the minimum price (in terms of minimum
distance) that one has to pay in order to enlarge $\mC$ to a bigger code. 
We can derive a precise relation
between the covering radius and the maximality
degree of a code as follows.

\begin{proposition} \label{murhod}
 For any code $\mC \subseteq \F_q^{k \times m}$ with $|\mC| \ge 2$ we have
$\mu(\mC)=d(\mC)-\min \{ \rho(\mC),\ d(\mC)\}$.
In particular, if $\mC$ is maximal then $\mu(\mC)=d(\mC)-\rho(\mC)$.
\end{proposition}

\begin{proof}
 If $\mC$ is not a maximal code, then by Proposition \ref{max}
 we have $\mu(\mC)=0$ 
and $\rho(\mC) \ge d(\mC)$. The result
immediately follows. 

Now assume that $\mC$ is maximal. 
If $\mC=\ma$ then the result is trivial. In the sequel we assume
$\mC \subsetneq \ma$.
By Proposition \ref{max}
we have $\min\{ \rho(\mC), \ d(\mC)\}=\rho(\mC)$.
We need to prove that 
\begin{equation*} \label{eqred}
 \mu(\mC)=d(\mC)-\rho(\mC).
\end{equation*}
Take $X \in \F_q^{k \times m} \setminus \mC$ with
$\min \{ d(X,M) : M \in \mC\}=\rho(\mC)$. Define the code $\mD:=\mC \cup
\{X\}\varsupsetneq \mC$.
By definition of minimum distance we have $d(\mD)=\min\{ d(\mC),\ \rho(\mC)\}=\rho(\mC)$, where the last 
equality again follows from Proposition \ref{max}.
As a consequence, $\mu(\mC) \le d(\mC)-d(\mD)=d(\mC)-\rho(\mC)$.
Now assume by contradiction that $\mu(\mC)<d(\mC)-\rho(\mC)$.
Let $\mD \subseteq \ma$ be a code with $\mD\varsupsetneq \mC$ and $d(\mC)-d(\mD)=\mu(\mC)$. We have
$d(\mC)-d(\mD)=\mu(\mC)<d(\mC)-\rho(\mC)$, and so $d(\mD)>\rho(\mC)$. This contradicts
Lemma \ref{rd}.
\end{proof}

\section{Puncturing and shortening rank-metric codes} \label{sec2}

In this section we propose new definitions of puncturing and shortening
of rank-metric codes, and show they relate to the minimum distance, the covering
radius and the duality theory of codes endowed with the rank metric.
Applications of our constructions will be discussed later.

\begin{notation}
  Given a code 
 $\mC \subseteq \ma$ and an integer $1 \le u \le k-1$, we let 
 $$\mC_u := \{M \in \mC : M_{ij}=0 \mbox{ whenever } i \le u\},$$
 the set of matrices in $\mC$ whose first $u$ rows are zero.
 Moreover, if $A$ is a $k \times k$ matrix over 
 $\F_q$ we define the code $A\mC:= \{A \cdot M : M \in \mC\} \subseteq \ma$. 
 Finally, $\pi_u:\ma \to \F_q^{(k-u) \times m}$ denotes the
projection 
 on the last $k-u$ rows. 
\end{notation}

Notice that
 if $A \in \mbox{GL}_k(\F_q)$ then the map $X \mapsto AX$
 is a linear rank-metric isometry 
 $\ma \to \ma$. In particular, if $\mC \subseteq \ma$ is a code, then 
 $A\mC$ is a code with the same cardinality, minimum distance,
 covering radius and weight and distance distribution as $\mC$. 


\begin{definition}
 Let $\mC \subseteq \ma$ be a code, $A \in \mbox{GL}_k(\F_q)$ an 
 invertible matrix and $1 \le u \le k-1$ a positive integer.
 The  \textbf{puncturing} of $\mC$ with respect to $A$ and $u$ is the code
  $$\Pi(\mC,A,u):=\pi_u(A\mC).$$
  When $0 \in \mC$, the \textbf{shortening}   of $\mC$ with
respect to $A$ and $u$ is the code
$$\Sigma(\mC,A,u):= \pi_u((A\mC)_u).$$
 \end{definition}
 
 The shortening and puncturing of a code 
 $\mC \subseteq \ma$ are codes in the ambient space $\F_q^{(k-u)\times m}$.
 Notice moreover that linearity is
preserved by puncturing and shortening.

 It will be convenient for us to use the following notation in the sequel.

\begin{notation}\label{notazU}
 Given a code $\mC \subseteq \F_q^{k\times m}$ and an
$\F_q$-linear subspace $U
\subseteq
\F_q^k$, we denote by $\mC(U)$ the 
set of matrices in $\mC$ whose columnspace is contained in the space $U$.
\end{notation}

\begin{remark}
	It is easy to see that if $\mC$ is linear, then $\mC(U)$ is an
	$\F_q$-linear subspace of $\mC$ for
	any $U$. 
	Moreover, if $U \subseteq \F_q^k$ is a given subspace of
dimension $u$, then 
$\mC_{k-u} \cong (A\mC)(U)$ as $\F_q$-linear spaces, where 
$A \in \F_q^{k \times k}$ is any invertible matrix that maps $\langle
e_{k-u+1},...,e_k \rangle$ to $U$ (here $\{e_1,...,e_k\}$ denotes the canonical
basis of $\F_q^k$).
\end{remark}

 We now show an interesting relation between puncturing, shortening, and
trace-duality.

\begin{theorem}[duality of puncturing and shortening] \label{duality}
 Let $\mC \subseteq \ma$ be a linear code, $A \in \mbox{GL}_k(\F_q)$ an
invertible matrix and $1 \le u \le k-1$ an integer. Then 
 $$\Pi(\mC,A,u)^\perp=\Sigma(\mC^\perp, (A^t)^{-1},u).$$
\end{theorem}

\begin{proof}
 Let $M \in \Sigma(\mC^\perp, (A^t)^{-1},u)=\pi_u(((A^t)^{-1}\mC^\perp)_u)$ and 
 $N \in \Pi(\mC,A,u)=\pi_u(A\mC)$.
 By definition, we can write  $N=\pi_u(AN_1)$ with $N_1 \in \mC$ and 
 $M=\pi_u((A^t)^{-1}M_1)$ with $M_1 \in \mC^\perp$ and $(A^t)^{-1}M_1 \in ((A^t)^{-1}\mC)_u$.
 Since the first $u$ rows of $(A^t)^{-1}M_1$ are zero, by definition of trace we have 
 $$\mbox{Tr}(\pi_u((A^t)^{-1}M_1) \pi_u(AN_1)^t) = \mbox{Tr} ((A^t)^{-1}M_1(AN_1)^t)=
 \mbox{Tr} ((A^t)^{-1}M_1N_1^tA^t)=\mbox{Tr}(M_1N_1^t)=0,$$
 where the last equality follows from the fact that $M_1 \in \mC^\perp$ and 
 $N_1 \in \mC$. This proves $(\supseteq)$.
  It suffices to show that the codes $\Pi(\mC,A,u)^\perp$ and 
 $\Sigma(\mC^\perp, (A^t)^{-1},u)$ have the same dimension over $\F_q$.
 Denote by 
 $\{e_1,...,e_k\}$ the canonical basis of $\F_q^k$, and let 
 $U:=\langle e_1,...,e_u\rangle$.
 One has 
 \begin{equation} \label{eqA}
  \dim (\Pi(\mC,A,u)^\perp) =
 m(k-u) - \dim (\Pi(\mC,A,u)) =
 m(k-u) - (\dim(\mC) - \dim ((A\mC)(U))),
 \end{equation}
 where the last equality  follows from the $\F_q$-isomorphism 
 $\Pi(\mC,A,u) \cong \mC / (A\mC)(U)$.
By \cite[Lemma 28]{alb} we have 
\begin{equation} \label{eqB}
 \dim ((A\mC)(U))) = \dim(A\mC) - m(k-u) + \dim((A\mC)^\perp (U^\perp)).
\end{equation}
Observe that $\dim(A\mC)=\dim(\mC)$ and $(A\mC)^\perp = (A^t)^{-1}\mC^\perp$.
Moreover, since $U^\perp=\langle e_{u+1},...,e_k \rangle$, by definition of shortening 
we have 
$\pi_u(((A^t)^{-1}\mC^\perp)(U^\perp)) = \Sigma(\mC^\perp, (A^t)^{-1},u)$. 
In particular, 
$\dim(\Sigma(\mC^\perp, (A^t)^{-1},u))=\dim(((A^t)^{-1}\mC^\perp)(U^\perp))$.
Thus Equation (\ref{eqB}) can be written as 
\begin{equation}\label{eqC}
 \dim ((A\mC)(U))) = \dim(\mC) - m(k-u) + \dim (\Sigma(\mC^\perp, (A^t)^{-1},u)).
\end{equation}
 Combining equations (\ref{eqA}) and (\ref{eqC}) we obtain
 $$\dim (\Pi(\mC,A,u)^\perp) = \dim (\Sigma(\mC^\perp, (A^t)^{-1},u)).$$
This concludes the proof.
\end{proof}

The following two propositions show how  
puncturing, shortening, cardinality, minimum distance and covering radius of 
rank-metric codes relate to each other. 

\begin{proposition}\label{generalpr}
 Let $\mC \subseteq \ma$ be a code with $|\mC| \ge 2$. 
 Let $A \in \mbox{GL}_k(\F_q)$ and $1 \le u \le k-1$.
 \begin{enumerate}
 \item $d(\Pi(\mC,A,u)) \ge d(\mC)-1$, if $|\Pi(\mC,A,u)| \ge 2$.
\label{pr1}
\item $d(\Sigma(\mC,A,u)) \ge d(\mC)$, if $0 \in \mC$ and
$|\Sigma(\mC,A,u)| \ge 2$.
\label{pr2}
 \item Assume $u \le d(\mC)-1$. Then $|\Pi(\mC,A,u)| = |\mC|$. 
If $\mC$ is linear, then $|\Sigma(\mC^\perp,A,u)| =
q^{m(k-u)}/|\mC|$.\label{pr3}
 \item Assume $u >
d(\mC)-1$. Then  
 $|\Pi(\mC,A,u)| \ge |\mC|/q^{m(u-d(\mC)+1)}$.  
If $0 \in \mC$, then $|\Sigma(\mC,A,k-u)| \leq q^{m(u-d(\mC)+1)}$.
\label{pr4}
 \end{enumerate}
\end{proposition}
\begin{proof}
 Properties \ref{pr1}, \ref{pr2}  are simple and left to the
reader. The first part of Property \ref{pr3}
follows from the definition of minimum distance, and the second part 
is a consequence of Theorem \ref{duality}.
Let us show Property \ref{pr4}.
Write $u=d(\mC)-1+v$ with $1 \le v \le
k-d(\mC)+1$, and define the code 
$\mE:=\Pi(\mC,A,d(\mC)-1)$.
By Property \ref{pr3} we have 
$|\mC|=|\Pi(\mC,A,d(\mC)-1|=|\mE|$.
It follows from the definitions that  
$\Pi(\mC,A,u)=\pi_v(\mE)$, where 
$$\pi_v:\F_q^{(k-d(\mC)+1) \times m} \to \F_q^{(k-u) \times m}$$
denotes the projection on the last $k-u$ rows.
For any $N \in \pi_v(\mE)$ let 
$[N]:=\{M \in \mE : \pi_v(M)=N\}$. Clearly,
$[N] \cap [N'] = \emptyset$ whenever $N,N' \in \pi_v(\mE)$ and $N \neq N'$.
Moreover, it is easy to see that $|[N]| \le q^{mv}$ for all
$N \in \pi_v(\mE)$.
Therefore
$$|\mE| \ = \ \left| \bigcup_{N \in \pi_v(\mE)} [N] \right| \ \ = \ \sum_{N \in
\pi_v(\mE)} |[N]| \ \le \ |\pi_v(\mE)|
\cdot q^{mv},$$
and so $|\Pi(\mC,A,u)|=|\pi_v(\mE)| \ge |\mE| / q^{mv}$.
Let us prove the last part of Property \ref{pr4}. If 
$|\Sigma(\mC,A,k-u)|=1$ then there is nothing to prove. Assume  
$|\Sigma(\mC,A,k-u)| \ge 2$. Then $\Sigma(\mC,A,k-u)$ has minimum distance at
least $d(A\mC)=d(\mC)$. Therefore by the Singleton-like bound \cite{D78} we have
$$|\Sigma(\mC,A,k-u)| \le q^{m(u-d(\mC)+1)},$$
as claimed.
\end{proof}

\begin{proposition}\label{boundCRpunct}
 Let $\mC \subseteq \ma$ be a code. For all $A \in \mbox{GL}_k(\F_q)$ and 
 $1 \le u \le k-1$ we have 
 $$\rho(\mC) \ge \rho(\Pi(\mC,A,u)) \ge \rho(\mC)-u.$$
 \end{proposition}

\begin{proof}
 Let $\mD:=A\mC$. Then
$\Pi(\mC,A,u)=\pi_u(\mD)$. 
 Let $X \in \ma$ be an arbitrary matrix. By definition of 
 covering radius and punctured code there exists $M \in \mD$ with 
 $d(\pi_u(M),\pi_u(X)) \le \rho(\pi_u(\mD)$. Therefore
 $d(M,X) \le d(\pi_u(M),\pi_u(X)) + u \le \rho(\pi_u(\mD)) +u$. 
 Since $X$ is arbitrary, this shows $\rho(\mD) \le \rho(\pi_u(\mD)) +u$, 
 i.e., $\rho(\pi_u(\mD)) \ge \rho(\mD)-u=\rho(\mC)-u$.
 
 Now let $X \in \F_q^{(k-u) \times m}$ be an arbitrary matrix. 
 Complete $X$ to a $k \times m$ matrix, say $X'$, by adding  
 $u$ zero rows to the top. There exists $M \in \mD$ with 
 $d(X',M) \le \rho(\mD)$. Thus
$$d(X,\pi_u(M))=d(\pi_u(X'),\pi_u(M))
 \le d(X',M) \le \rho(\mD)=\rho(\mC).$$
 This shows $\rho(\pi_u(\mD))  \le \rho(\mC)$, and concludes the proof.
\end{proof}

\section{Translates of a rank-metric code} \label{sec3}
In this section we study the weight distribution of the translates of a code.
As an application, we obtain two upper bound on the covering radius of 
a rank-metric code. Recall that the \textbf{translate} of a code $\mC \subseteq \ma$ by a matrix 
$X \in \ma$ is the code 
$$\mC+X:=\{M+X : M \in \mC\} \subseteq \ma.$$

Clearly, full knowledge of the weight distribution of the translates of $\mC$ tells 
us the covering radius, which is the maximum of the minimum weight 
of each translate of $\mC$. Even partial information may yield a bound on the covering radius. 
More precisely, if $X \in \ma$ and $W_i(\mC+X) \neq 0$, then
$d(X,\cC) := \min \{ d(X,M): M \in \cC \} \leq i$. 
So if there exists $r$ such that for each $X \in \fq^{k \times m}$, $W_i(\mC+X) \neq 0$ for 
some $i\leq r$ then, in particular,
$\rho(\mC) \leq r$. If such a value $r$ can be determined, 
then we get an upper bound on the covering radius of $\mC$. 

The goal of this section is twofold. We first show that the weight distribution 
$W_0(\mC+X),...,W_k(\mC+X)$ of the translate $\mC+X$ of a linear code
$\mC \subsetneq \F_q^{k\times m}$ is determined by the values of  
$W_0(\mC+X),...,W_{k-d^\perp}(\mC+X)$, where $d^\perp=d(\mC^\perp)$.
Moreover, we provide explicit formulas for  
$W_{k-d^\perp+1}(\mC+X),...,W_{k}(\mC+X)$ as linear functions of 
$W_0(\mC+X),...,W_{k-d^\perp}(\mC+X)$. As a simple application, we obtain 
an upper bound on the covering radius of a linear code in terms 
of the minimum distance of its dual code. Our proof uses 
combinatorial methods partly inspired by the 
theory of regular support functions on groups developed in \cite{alblatt}.

In a second part, following work of Delsarte for the Hamming distance \cite{D73}, we apply
Fourier transform methods to obtain further results on the
weight distributions of the translates of a (not necessarily linear) code $\cC \subseteq \fq^{k \times m}$.
In particular, we obtain an upper bound for the covering radius of 
a general rank-metric code in terms of its external distance (defined below).

Throughout this section we follow Notation 
\ref{notazU}. We start with a preliminary lemma 
that describes some combinatorial properties of the translates of a linear
code.

\begin{lemma} \label{suff}
 Let $\mC \subseteq \F_q^{k \times m}$ be a linear code, and let 
 $U \subseteq \F_q^k$ be an $\F_q$-linear subspace of dimension $u$. Assume
that 
 $|\mC(U)|=|\mC|/q^{m(k-u)}$. Then 
 for all matrices $X \in \F_q^{k \times m}$ we have 
 $$|(\mC+X)(U)|=|\mC|/q^{m(k-u)}.$$ 
\end{lemma}

\begin{proof}
Let 
$f:\F_q^k \to
\F_q^k$
be a linear isomorphism such that 
 $f(U)=V:=\{(x_1,...,x_k) \in \F_q^k : x_i =0 \mbox { for all } i >u \}$.
 Let $A$ be the matrix associated to $f$ with respect to the canonical basis of
$\F_q^k$. Define the linear code $\mD:=A\mC$. The
left-multiplication by $A$ induces bijections 
$$\mC(U) \to \mD(V), \ \ \ \ \ \ \ \ (\mC+X)(U) \to (\mD +AX)(V).$$
In particular, we have $|\mD(V)|=|\mC(U)|$, and it suffices to prove
that 
\begin{equation} \label{toprove}
 |(\mD+AX)(V)|= |\mD(V)|.
\end{equation}
Let $\pi:=\pi_u:\F_q^{k\times m} \to \F_q^{(k-u)\times m}$ denote the projection on the last 
$k-u$ rows. Throughout the proof we denote by $\pi_1$ and $\pi_2$ the restriction
of $\pi$ to $\mD$ and to $\mD+AX$, respectively. Clearly, $\pi_1$ is linear.

By definition of $V$ we have 
$\ker(\pi_1)=\mD(V)$. Therefore 
$$|\pi_1(\mD(V))|=|\mD| / |\mD(V)|=|\mC| / |\mC(U)|=q^{m(k-u)}.$$
In particular,  
$\pi_1$ is surjective.
Again by definition of $V$, we have  
$(\mD+AX)(V)=\pi_2^{-1}(0)$. Moreover, one can check that
$|\pi_2^{-1}(0)|=|\pi_1^{-1}(-\pi(AX))|$. Thus
\begin{equation} \label{uss}
 |(\mD+AX)(V)|=|\pi_2^{-1}(0)|=|\pi_1^{-1}(-\pi(AX))|.
\end{equation}
Since $\pi_1$ is surjective,
there exists $N \in \mD$ such that $\pi_1(N)=-\pi(AX)$. One can easily check
that
the
map 
$\ker(\pi_1) \to \pi_1^{-1}(-\pi(AX))$ defined by $M \mapsto M+N$ is a
bijection.
Thus using Equation (\ref{uss}) and the fact that $\mD(V)=\ker(\pi_1)$ we find 
$$|\mD(V)|=|\ker(\pi_1)|=|\pi_1^{-1}(-\pi(AX))|=|(\mD+AX)(V)|.$$ This shows
Equation (\ref{toprove}), as desired.
\end{proof}

 A second  preliminary result which will be needed later is the
following.
 
 \begin{lemma} \label{propprel}
  Let $\mC \subsetneq \F_q^{k\times m}$ be a linear code.
  Then for all matrices $X \in \F_q^{k\times m}$ and for any subspace 
  $U \subseteq \F_q$ with $u:=\dim(U) \ge k-d(\mC^\perp)+1$ we have 
  $$(\mC+X)(U)= |\mC|/q^{m(k-u)}.$$
   \end{lemma}

\begin{proof}
By Lemma \ref{suff} it suffices to prove the result for $X=0$. By 
\cite[Lemma 28]{alb}, for any subspace $U \subseteq \F_q^k$ of dimension $u$ we
have 
\begin{equation}\label{eeee}
 |\mC(U)| = \frac{|\mC|}{q^{m(k-u)}} |\mC^\perp(U^\perp)|,
\end{equation}
where $U^\perp$ denotes the orthogonal of $U$ with respect to the standard inner
product of $\F_q^k$. By definition of minimum distance we have 
$\mC^\perp(U^\perp)=\{0\}$ for all $U \subseteq \F_q^k$ with $\dim(U^\perp) \le
d(\mC^\perp)-1$. Therefore the lemma immediately follows
from Equation (\ref{eeee}) and the fact that 
$\dim(U^\perp)=k-\dim(U)$.
\end{proof}

We can now state our main result on the weight distribution of the 
translates of a linear rank-metric code.

\begin{theorem}\label{fff}
 Let $\mC \subsetneq \F_q^{k\times m}$ be a linear code, and let $X \in
\F_q^{k\times m}$ be any matrix. Write 
$d^\perp:=d(\mC^\perp)$. Then for all 
$i \in \{ k-d^\perp+1,...,k \}$ we have 
\begin{equation*}
 W_i(\mC+X)=\sum_{u=0}^{k-d^\perp} {(-1)}^{i-u} q^{\binom{i-u}{2}}
\qbin{k-u}{i-u}{q}\sum_{j=0}^u W_j(\mC+X) \qbin{k-j}{u-j}{q} +
\sum_{u=k-d^\perp+1}^i\qbin{k}{u}{q} \frac{|\mC|}{q^{m(k-u)}}.
\end{equation*}
In particular, the distance distribution of the translate $\mC+X$ is completely determined 
by $k$, $m$, $|\mC|$ and the weights $W_0(\mC+X),...,W_{k-d^\perp}(\mC+X)$.
\end{theorem}

\begin{proof}
Recall from \cite{ec} that the set of subspaces of $\F_q^k$ is a graded lattice
with respect to the partial order given by the inclusion. The rank
function of
this lattice is the dimension of vector spaces, and its  M\"{o}bius function is
given by 
$$\mu(S,T)={(-1)}^{t-s} q^{\binom{t-s}{2}}$$ for all
subspaces $S \subseteq T \subseteq \F_q^k$ with $\dim(T)=t$ and $\dim(S)=s$.
More details can be found on page 317 of \cite{ec}. Throughout the
proof a sum over an
empty set of indices is zero
by definition. 
 For any subspace $V \subseteq \F_q^k$ define 
$$f(V):=|\{ M \in \mC+X : \mbox{columnspace}(M) = V\}| \ \ \ \mbox{and} \ \ \  
g(V):=\sum_{U \subseteq V} f(V)= |(\mC+X)(V)|.$$
By the M\"{o}bius inversion formula (\cite{ec}, Proposition 3.7.1), for
any subspace $V \subseteq \F_q^k$ we have 
\begin{equation} \label{orig}
 f(V)= \sum_{U \subseteq V} |(\mC+X)(U)| \ \mu(U,V).
\end{equation}
Fix any integer $i$ with $k-d^\perp +1 \le i \le k$. By definition of weight 
distribution we have 
\begin{equation*}
 W_i(\mC+X)= \sum_{\substack{V \subseteq \F_q^k \\ \dim(V)=i}} f(V).
\end{equation*}
Therefore by Equation (\ref{orig}) the number $W_i(\mC+X)$
can be expressed as 
\allowdisplaybreaks
\begin{eqnarray}  \label{lunga}
W_i(\mC+X) &=& \sum_{\substack{V \subseteq \F_q^n \\ \dim(V)=i}} \sum_{U \subseteq
V} |(\mC+X)(U)| \ \mu(U,V) \nonumber \\ 
&=& \sum_{U \subseteq \F_q^k} \sum_{\substack{V \supseteq U \\ \dim(V)=i}}
|(\mC+X)(U)| \ \mu(U,V) \nonumber\\ 
&=& \sum_{U \subseteq \F_q^k} |(\mC+X)(U)| \ \sum_{\substack{V \supseteq U \\
\dim(V)=i}} \mu(U,V) \nonumber\\
&=& \sum_{u=0}^i \sum_{\substack{U \subseteq \F_q^k \\ \dim(U)=u}} |(\mC +
X)(U)| \sum_{\substack{V \supseteq U \\
\dim(V)=i}} \mu(U,V) \nonumber\\
&=& \sum_{u=0}^i \sum_{\substack{U \subseteq \F_q^k \\ \dim(U)=u}} |(\mC +
X)(U)| \sum_{\substack{V \supseteq U \\
\dim(V)=i}} {(-1)}^{i-u} q^{\binom{i-u}{2}} \nonumber\\
&=& \sum_{u=0}^i {(-1)}^{i-u} q^{\binom{i-u}{2}} \qbin{k-u}{i-u}{q}
\sum_{\substack{U \subseteq \F_q^k \\ \dim(U)=u}} |(\mC +
X)(U)| 
\end{eqnarray}
We now re-write the quantity  
$$\sum_{\substack{U \subseteq \F_q^k \\ \dim(U)=u}} |(\mC +
X)(U)|$$
in a more convenient form. By Lemma \ref{propprel}, for $u \ge k-d^\perp
+1$ we
have 
\begin{equation} \label{aa}
 \sum_{\substack{U \subseteq \F_q^k \\ \dim(U)=u}} |(\mC +
X)(U)| = \qbin{k}{u}{q} |\mC|/q^{m(k-u)}.
\end{equation}
On the other hand, for $u \le k-d^\perp$ we have 
\allowdisplaybreaks
\begin{eqnarray}  \label{bb}
 \sum_{\substack{U \subseteq \F_q^k \\ \dim(U)=u}} |(\mC +
X)(U)| &=& |\{ (U,M) : U \subseteq \F_q^k, \ \dim(U)=u, \ M \in \mC+X,
\ \mbox{columnspace}(M) \subseteq U\}| \nonumber \\
&=& \sum_{M \in \mC+X} |\{ U \subseteq \F_q^k : \dim(U)=u, \ U \supseteq
\mbox{columnspace}(M)\}| \nonumber \\
&=& \sum_{j=0}^u \ \ \sum_{\substack{M \in \mC+X \\ \textnormal{rk}(M)=j}} |\{ U
\subseteq \F_q^k : \dim(U)=u, \ U \supseteq
\mbox{columnspace}(M)\}| \nonumber\\
&=& \sum_{j=0}^u W_j(\mC+X) \qbin{k-j}{u-j}{q}.
\end{eqnarray}
Combining equations (\ref{lunga}), (\ref{aa}) and (\ref{bb}) one obtains 
the desired formula.
\end{proof}

 As a simple consequence of Theorem \ref{fff} we can obtain 
 an upper bound on the covering radius
of a linear code $\mC \subsetneq \F_q^{k\times m}$ in terms of its dual distance, as we now show. 
Let $X \in
\F_q^{k \times m} \notin \mC$. Then $W_0(\mC+X)=0$. Theorem \ref{fff}
with $i:=k-d^\perp +1$ gives
\begin{equation*}
 W_{k+d^\perp+1}(\mC+X)=\sum_{u=1}^{k-d^\perp} {(-1)}^{i-u} q^{\binom{i-u}{2}}
\qbin{k-u}{i-u}{q}\sum_{j=1}^u W_j(\mC+X) \qbin{k-j}{u-j}{q} +
\qbin{k}{k-d^\perp+1}{q} |\mC|/q^{m(d^\perp-1)}.
\end{equation*}
In particular, $W_1(\mC+X),...,W_{k-d^\perp+1}(\mC+X)$ cannot be all zero. This 
implies the following.

\begin{corollary}[dual distance bound]\label{dualbound}
 For any linear code $\mC \subsetneq \ma$ we have $\rho(\mC) \le k-d(\mC^\perp) +1$.
\end{corollary}

We now relate the covering radius of a 
code with its external distance.
In particular, we derive another upper bound on the covering radius of a linear code in
terms of the rank distribution of the dual code. 
This is the rank distance analogue of Delsarte's external Hamming distance bound (c.f. \cite{McWS,D73,CKMS85}),
and improves the dual distance bound of Corollary \ref{dualbound}.

The approach uses $q$-Krawtchouk polynomials and Fourier transforms to obtain relations on the weight
distribution of the translates of a code in $\fq^{k \times m}$. The properties 
of $q$-Krawtchouk polynomials were described in \cite{D73,D76}.
The Fourier transform arguments used are independent of 
the choice of metric used and so extend from the Hamming metric case. 
The principle novelty is the introduction of a $q$-annihilator polynomial, 
used in the proof of Lemma \ref{lemannpoly}.

Throughout the reminder of this section 
$\mC \subseteq \ma$ denotes a (possibly non-linear) code, and 
$\chi$ is a fixed non-trivial character of $(\fq,+)$.

\begin{definition}
Let $Y \in \fq^{k \times m}$. Define the \textbf{character map} on $(\fq^{k
\times m},+)$ associated to $Y$ by
$$\phi_Y:\fq^{k \times m} \longrightarrow \C^\times : X \mapsto \chi(\tr(Y X^T)).$$
\end{definition}

Clearly $\phi_X(Y) = \phi_Y(X)$ for all $X,Y \in \fq^{k \times m}.$
We denote by $\Phi$ the $km \times km$ symmetric matrix with values 
in $\C^\times$ defined as having entry $\phi_Y(X)$ in the column indexed by $X$ and in the row indexed by $Y$.  
Define the $\Q$-module of length $km$:
${\mathfrak C} := \left\{ (\mA_X: X\in \fq^{k\times m}) : \mA_X \in \Q \right\}.$
For each $Y$, extend $\phi_Y$ to a character of ${\mathfrak C}$ as follows:
$$ \phi_Y: {{\mathfrak C}} \longrightarrow \C^\times : \mA=(\mA_X: X \in \fq^{k \times m}) \mapsto \sum_{X} \mA_X \phi_Y(X).$$
Then $\Phi \mA= (\phi_Y(\mA): Y \in  \fq^{k \times m}) \in {\mathfrak C}.$
The rows of $\Phi$ are pairwise orthogonal, as can be seen from:
$$ \sum_X \phi_Y(X) \phi_Z(X) = \sum_X \phi_X(Y) \phi_X(Z)=
\sum_X \phi_X(Y-Z) = \sum_X \phi_{Y-Z}(X) 
= \left\{ \begin{array}{ll} 
q^{km} & \text{if } Y=Z ,\\
0 & \text{otherwise.}
\end{array}\right.$$ 
Therefore $\Phi^2 \mA = \Phi^T \Phi \mA = q^{km} \mA$ and so $\mA$ is determined completely by its 
{transform} 
$$\mA^*:=\Phi \mA=(\phi_Y(\mA): Y \in \fq^{k \times m}).$$

Any subset ${\mathcal{ U}} \subseteq \fq^{k \times m}$ can be identified with
the $0$-$1$ vector 
$\overline{{\mathcal U}}=({\mathcal U}_Z: Z \in \fq^{k\times m})\in {\mathfrak C}$, where 
$${\mathcal{U}}_Z = \left\{ \begin{array}{ll} 
1 & \text{if } Z \in {\mathcal U},\\
0 & \text{otherwise.}
\end{array}\right.$$
For any $X \in \fq^{k \times m}$, the translate code ${\mC} +X 
\subseteq \fq^{k \times m}$ is  then
identified with $\overline{{\mC}+X}=({\mC}_{Z-X}: Z \in \fq^{k \times m}).$
It is straightforward to show that $\phi_Y(\overline{\cC+X}) = \phi_Y(\overline{\cC})\phi_Y(X)$. 
This immediately yields the inversion formula
$${\mC}_X = \frac{1}{q^{km}} \sum_Y \phi_Y (\overline{\mC + X})= 
\frac{1}{q^{km}} \sum_Y \phi_Y (\overline{{\mathcal C}})\phi_Y(X). $$

For each $ i \in [k]$ we let $\Omega^i$ be the set of matrices in $\fq^{k \times m}$ of rank $i$.  

\begin{lemma}[see \cite{D78}] Let $Y \in \fq^{k \times m}$. Then $\phi_Y (\overline{\Omega^i})$ 
	depends only on the rank of $Y$. If $Y$ has rank $j$, then  this is given by 
$$P_i(j)  :=  \sum_{\ell=0}^k (-1)^{i-\ell} q^{\ell m +
			\binom{i-\ell}{2}} \qbin{k-\ell}{k-i}{q}\qbin{k-j}{\ell }{q}.$$
\end{lemma}	         
In terms of the transform of $\Omega^i$ this gives
$$\Phi \overline{\Omega^i} = (P_i({\rk(Y)}): Y \in \fq^{k \times m}). $$
It is known \cite{D76,D78} that the $P_i(j)$ are orthogonal polynomials of degree $i$ in the variable $q^{-j}$.
Therefore, any rational polynomial $\gamma$ of degree at most $k$ in
$q^{-j}$ can be expressed as a $\Q$-linear combination of the $q$-Krawtchouck
polynomials:
$\gamma(x) =\sum_{j=0}^{k} \gamma_j P_j(x)$. Again, the orthogonality relations mean that the 
coefficients can be of $\gamma$ can be retrieved as $$\gamma_j=\frac{1}{q^{km}}\sum_{i=0}^{k}\gamma(i)P_i(j).$$
We let $P=(P_i(j))$ denote the $(k+1) \times (k+1)$ matrix with $(j,i)$-th component equal to $P_i(j)$.
Then the \textbf{transform} \label{deftransform}of $B(\mC) = (B_i(\cC) : 0 \leq i \leq k)$ is defined as
$B^*(\mC):=|\cC|^{-1} B(\mC)P$.
The coefficents of $B^*(\mC)$ are non-negative \cite[Theorem 3.2]{D78}.

Let $\mD:=(D_Z: Z \in \fq^{k \times m})$ where $D_Z = |\{ (X,Y) : X,Y \in \cC, X+Y = Z\}|$.
It can be checked that
$$\phi_Y(\mD) = \phi_Y(\mC)\phi_Y(\mC) = \phi_Y(\mC)^2.$$
Then
$$\sum_{Y \in \Omega^i} \phi_Y(\mD) = \sum_Z D_Z \sum_{Y \in \Omega^i} \phi_Y(Z) = 
\sum_Z D_Z \phi_Z(\Omega^i) = \sum_Z D_Z P_i(\rk (Z)) 
= |\cC|\sum_{j=0}^k B_j(\mC) P_i(j) = |\cC|(B(\mC)P)_i, $$
and in particular we have
$$|\mC| B^*(\mC)= (\sum_{Y \in \Omega^i} \phi_Y(\mD) : 0 \leq i \leq k) = 
(\sum_{Y \in \Omega^i} \phi_Y(\mC)^2 : 0 \leq i \leq k). $$
Clearly $B_i^*(\mC)=0$
implies that $\phi_Y(\cC)=0$ for each $Y \in \Omega^i$.

\begin{definition}
The \textbf{external distance} of a code $\cC \subseteq \ma$ is the integer
$$\sigma^*(\cC):=|\{i \in [k] : B_i^*(\mC) > 0\}|,$$ 
the number of non-zero coefficents of $B^*(\mC)$, excluding $B_0^*(\mC)$ .
\end{definition}

For ease of notation in the sequel we write $\sigma^*:=\sigma^*(\cC)$. Let
$0<b_1<...<b_{\sigma^*}\leq k$ denote the indices $i$ of non-zero
$B_i^*(\mC)$ for $i>0$.

\begin{definition}	The \textbf{annihilator polynomial} of degree $\sigma^*$ in the
variable $q^{-x}$ of $\cC$ is
	$$\alpha(x): =
	\frac{q^{mn}}{|\cC|}\prod_{j=1}^{\sigma^*}\frac{1-q^{b_j-x}}{1-q^{b_j}
	} = \sum_{j=0}^{\sigma^*} \alpha_j P_j(x) .$$
\end{definition}

This is the $q$-analogue of the Hamming metric annihilator 
polynomial \cite[pg. 168]{McWS}. Notice that 
the $b_j$ are the zeroes of $\alpha$ and $\alpha(0) = \frac{q^{mn}}{|\cC|}$.

\begin{lemma}\label{lemannpoly} Let $X \in \fq^{ k \times m}$ be an arbitrary matrix. Then
	$$\sum_{j=1}^{\sigma^*} \alpha_j W_j(\mC + X) = 1.$$
	In particular, there exists some $j \in [\sigma^*]$ such that $W_j(\mC+X) > 0$.
\end{lemma}

\begin{proof}

We must show that $\sum_{j=1}^{\sigma^*} 
\alpha_j( W_j(\mC + X):X \in \fq^{k \times m}) = (1: X \in \fq^{k \times m})$.
Since $\Phi$ is invertible, this holds if and only if for all $Y$, 
$$ \phi_Y \left(\sum_{j=1}^{\sigma^*} \alpha_j W_j(\mC+X):X \in \fq^{k \times
}\right) = 
\phi_Y(1: X \in \fq^{k \times m}) =
\left\{ \begin{array}{ll} 0   & \text{ if } Y \neq 0 \\
q^{km} & \text{ if } Y = 0 .\\ \end{array}                           
\right. $$
This was the approach taken, for example, in \cite[Chapter 6, Lemma 18]{McWS}.
Let $Y \in \fq^{k \times m}$. Then
\allowdisplaybreaks
\begin{eqnarray*}
		\phi_Y \left(\sum_{j=1}^{\sigma^*} \alpha_j W_j(\mC+X):X \in \fq^{k \times m}\right) 
		& = &  \sum_{j=1}^{\sigma^*} \alpha_j \sum_X  W_j(\mC+X) \phi_Y(X)\\	
		& = & \sum_{j=1}^{\sigma^*} \alpha_j \sum_{X \in \Omega^j} \phi_Y(\overline{\mC+X})\\
		& = & \sum_{j=1}^{\sigma^*} \alpha_j \sum_{X \in \Omega^j} \phi_Y(\overline{\cC})\phi_Y(X)\\
		& = & \sum_{j=1}^{\sigma^*} \alpha_j \phi_Y(\overline{\cC}) \sum_{X \in \Omega^j} \phi_Y(X)\\
		& = & \sum_{j=1}^{\sigma^*} \alpha_j \phi_Y(\overline{\cC})  \phi_Y(\overline{\Omega^j})\\
		& = & \sum_{j=1}^{\sigma^*} \alpha_j \phi_Y(\overline{\cC}) P_j(\rk(Y)) \\
		& = & \phi_Y(\overline{\cC}) \sum_{j=1}^{\sigma^*} \alpha_j P_j(\rk(Y))\\
		& = & \phi_Y(\overline{\cC}) \alpha(\ell),		
	\end{eqnarray*}
	where $Y$ has rank $\ell$.
	
	Now $\alpha(0) =\frac{q^{mn}}{|\cC|} $ and
	$\phi_0(\cC)=|\cC|$, so $\phi_0(\overline{\cC}) \alpha(0)=q^{km}$. 
	Suppose that $Y$ has rank $\ell>0$. The roots of $\alpha$ are precisely those $j \geq 1$ such that $B_j^*(\mC)$ 
	is non-zero.
	On the other hand, if $B_j^*(\mC)=0$ then $\phi_Y(\overline{\cC})=0$.
	It follows that the product $\phi_Y(\overline{\cC}) \alpha(\ell)=0$ and so 
	$$\Phi (\sum_{j=1}^{\sigma^*} \alpha_j W_j(\mC+X): X \in \fq^{k \times m}) = 
	\Phi (1 : X \in \fq^{k \times m}),$$
	as claimed.
\end{proof}

We can now upper-bound the covering radius of a general rank-metric code in
terms
of its external distance as follows.

\begin{theorem}[external distance bound] \label{ei}
	For any code $\mC \subseteq \fq^{m\times n}$ we have 
	$\rho(C) \leq \sigma^*(\mC).$
	Furthermore, if $\mC$ is $\fq$-linear then $\rho(\mC)$ is no greater than
	the number of non-zero weights of $\mC^\perp$, excluding $W_0(\mC^\perp)$.
\end{theorem}	

\begin{proof}
 The first part of the theorem is an immediate consequence of Lemma 
 \ref{lemannpoly}. The second part follows from the fact that 
 $B_i^*(\mC)=B_i(\mC^\perp)=W_i(\mC^\perp)$, provided that 
 $\mC$ is linear. This can be easily seen from the definition 
 of $B^*(\mC)$ on page \pageref{deftransform} and the MacWilliams identities for the rank metric
 \cite{D78}.
\end{proof}

\begin{example}
 	Let $m=rs$ and let $\mC=\{\sum_{i=0}^{r-1} f_i x^{q^{si}} : f_i \in
\ff_{q^m}  \} .$
 	Then $\mC$ is the set of all $\ff_{q^s}$-linear maps from $ \ff_{q^m} $
to itself.
 	Therefore $\mC$ has elements of $\ff_{q^s}$-ranks $0,1,2,...,r$.
 	Let $f$ have rank $i$ over $\ff_{q^s}$. Let Im $f \subseteq \ff_{q^m}$
have $\ff_{q^s}$-basis $\{v_1,...,v_i\}$ and let
 	$\{u_1,...,u_s\}$ be an $\fq$-basis of $\ff_{q^s}$. Then $\{u_iv_j :1
\leq i \leq s, \ 1 \leq j \leq i \}$ is an
 	$\fq$-basis of Im $f$ in $\ff_{q^m}$, and so has dimension $is$.    
 	Then $\cC$ has non-zero rank weights $\{s,2s,...,rs\}$ over $\fq$, so that
 	$\rho(\mC^\perp) \leq r$. 
 \end{example}

\section{Initial set bound} \label{sec4}

In this section we propose a definition of initial set of 
a linear rank-metric code inspired by \cite{mesh}. Moreover we exploit the
combinatorial structure of such set 
to derive an upper bound for the covering radius of the underlying
code. Our technique relies on the specific ``matrix structure'' of 
rank-metric codes.

\begin{notation}
Given positive integers $a,b$ and a set $S \subseteq [a] \times [b]$, 
we denote by $\numberset{I}(S) \in \F_2^{a \times b}$ be the
binary matrix defined by
$\numberset{I}(S)_{ij}:=1$ if $(i,j) \in S$, and 
$\numberset{I}(S):=0$ if $(i,j) \notin S$. Moreover, 
 we denote by $\lambda(S)$ the minimum number of lines (rows or columns) 
required to cover all the ones in $\numberset{I}(S)$.
\end{notation}

 The initial set of a linear code is defined as follows.

\begin{definition}
 Let $\preceq$ denote the lexicographic order on $[k]\times [m]$. 
 The \textbf{initial entry} of a non-zero matrix $M \in \F_q^{k \times m}$ is
$\textnormal{in}(M):=
\min_{\preceq}\{ (i,j) : M_{ij} \neq 0\}$. The \textbf{initial set} of 
a non-zero linear code $\mC \subseteq \F_q^{k \times m}$ is
$$\textnormal{in}(\mC):=\{ \mbox{in}(M) : M \in \mC, \ M \neq 0\}.$$
\end{definition}

We start with a preliminary lemma.

\begin{lemma}
 Let $\mC \subseteq \F_q^{k \times m}$ be a non-zero code. The following hold.
\begin{enumerate}
 \item $\dim(\mC)=|\mbox{in}(\mC)|$,
\item $\mbox{in}(\mC) \subseteq [k-d(\mC)+1] \times [m]$.
\end{enumerate}
\end{lemma}

\begin{proof}
 Let $t:=\dim(\mC)$, and let $\{ M_1,...,M_t\}$ be a basis of $\mC$. 
Without loss of 
generality we may assume
$(1,1) \preceq \mbox{in}(M_1) \prec \cdots \prec \mbox{in}(M_t)$.
If $M \in \mC \setminus \{0\}$, then there exist elements $a_1,...,a_t \in \F_q$ such that
$M=\sum_{i=1}^t a_iM_i$, hence $\mbox{in}(M) \in \{
\mbox{in}(M_1),...,\mbox{in}(M_t)\}$.
This shows $\mbox{in}(\mC)=\{ \mbox{in}(M_1),...,\mbox{in}(M_t)\}$.
In particular, $|\mbox{in}(\mC)|=t=\dim(\mC)$.
Notice moreover that if $\mbox{in}(M_t) \succ (k-d(\mC)+1,m)$, then 
clearly $\mbox{rk}(M_t) \le d(\mC)-1$, a contradiction. Therefore we have
$$(1,1) \preceq \mbox{in}(M_1) \prec \cdots \prec \mbox{in}(M_t) \preceq
(k-d(\mC)+1,m).$$
This shows $\mbox{in}(\mC) \subseteq [k-d(\mC)+1] \times [m]$.
\end{proof}

\begin{remark} \label{rrrmmm}
 Let $a,b$ be positive integers and let $S \subseteq [a] \times [b]$ be a set.
 Assume that $M \in \F_q^{a \times b}$ is a matrix with $M_{ij}=0$ whenever 
$(i,j) \notin S$. Then $\mbox{rk}(M) \le \lambda(S)$.
This can be proved by induction on $\lambda(S)$.
\end{remark}

We can now state the main result of this section, which 
provides an upper bound on the covering
radius of a linear rank-metric code  $\mC$ in terms of the combinatorial structure of its
initial set.

\begin{theorem}[initial set bound] \label{nubound}
 Let $\mC \subseteq \F_q^{k \times m}$ be a non-zero linear code. We have 
 $\rho(\mC) \le d(\mC)-1+\lambda(S)$, where 
   $S:= [k-d(\mC)+1] \times [m] \setminus \mbox{in}(\mC)$. 
   \end{theorem}

\begin{proof}
Let $X \in \F_q^{k \times m}$ be any matrix. It is easy to see that there exists a 
unique matrix $M \in \mC$ such that
$X_{ij}=M_{ij}$ for all $(i,j) \in \mbox{in}(\mC)$. Such matrix satisfies 
$(X-M)_{ij}=0$ for all $(i,j) \in \mbox{in}(\mC)$. Let $\overline{X-M}$ be the
matrix obtained
from $X-M$ deleting the last $d(\mC)-1$ rows. We have 
$$d(X,M)=\mbox{rk}(X-M) \le d(\mC)-1+\mbox{rk}(\overline{X-M}) \le
d(\mC)-1+\lambda(S),$$
where $S$ denotes the complement of $\mbox{in}(\mC)$ in
$[k-d(\mC)+1] \times [m]$, and the last inequality follows from 
Remark
\ref{rrrmmm}. Since $X$ is an arbitrary matrix, this 
shows 
$\rho(\mC) \le d(\mC)-1+\lambda(S)$.
\end{proof}

\begin{remark}
 The  initial set of a linear code $\mC \subseteq \F_q^{k \times m}$ can be efficiently
computed from any basis of $\mC$ as follows.
Denote by $w:\F_q^{k \times m} \to \F_q^{mk}$ the map that sends a matrix $M$
to the $mk$-vector obtained concatenating the rows of $M$.
Given a basis $\{M_1,...,M_t\}$ of $\mC$, construct the vectors 
$v_1:=w(M_1),...,v_t:=w(M_t)$. Perform Gaussian elimination on
$\{v_1,...,v_t\}$ and obtain vectors $\overline{v}_1,...,\overline{v}_t$.
Clearly, $\{ w^{-1}(\overline{v}_1),...,w^{-1}(\overline{v}_t)\}$ is a basis
of $\mC$, and one can easily check that
$$\mbox{in}(\mC)=\{
\mbox{in}(w^{-1}(\overline{v}_1)),...,\mbox{in}(w^{-1}(\overline{v}_t))
\}.$$
\end{remark}

The following example shows that Theorem \ref{nubound} gives in some cases 
a better bound than Corollary \ref{ei} for the covering radius of a linear code.

\begin{example}
 Let $q=2$ and $k=m=3$. Denote by $\mC$ the linear code generated over $\F_2$ by the 
four 
matrices
$$\begin{bmatrix}
   1 & 0 & 0 \\ 0 & 0 & 1 \\ 0 & 0 & 0
  \end{bmatrix}, \ \ \ 
\begin{bmatrix}
   0 & 1 & 0 \\ 0 & 0 & 0 \\ 1 & 0 & 0
  \end{bmatrix}, \ \ \ 
\begin{bmatrix}
   0 & 0 & 0 \\ 1 & 0 & 0 \\ 0 & 1 & 0
  \end{bmatrix}, \ \ \ 
\begin{bmatrix}
   0 & 0 & 0 \\ 0 & 1 & 1 \\ 1 & 0 & 0
  \end{bmatrix}.
$$
We have $d(\mC)=2$. Moreover, since
$$\begin{bmatrix}
   0 & 0 & 1 \\ 0 & 0 & 0 \\ 0 & 0 & 0
  \end{bmatrix},  \begin{bmatrix}
   0 & 0 & 0 \\ 1 & 0 & 0 \\ 0 & 1 & 0
  \end{bmatrix}, \begin{bmatrix}
   0 & 0 & 1 \\ 1 & 0 & 0 \\ 0 & 1 & 0
  \end{bmatrix} \in \mC^\perp,
$$
we have $\sigma(\mC^\perp)=3$, and so Corollary \ref{ei} gives $\rho(\mC) \le 3$.
On the other hand, one can easily check that the initial set of $\mC$ is 
$\mbox{in}(\mC)=\{ (1,1), (1,2), (2,1), (2,2)\}$. Thus following the notation of 
Theorem \ref{nubound} we have $S=\{ (1,3), (2,3)\}$ and 
$\lambda(S)=1$. It follows 
$\rho(\mC) \le d(\mC)-1+\lambda(S)=2$. Therefore Theorem \ref{nubound} gives a
better bound on $\rho(\mC)$ than Corollary \ref{ei}. In fact, one can check that 
$\rho(\mC)=2$.
\end{example}

\section{Covering radius of MRD and dually QMRD codes} \label{sec5}

It is well known \cite{D78} that if $\mC \subseteq \ma$ is a code with 
$|\mC| \ge 2$, then $\log_q |\mC| \le m(k-d(\mC)+1)$.
A code $\mC \subseteq \ma$ is \textbf{MRD} if $|\mC|=1$ or 
$|\mC| \ge 2$ and $\log_q |\mC|=m(k-d(\mC)+1)$.  
MRD codes have the largest possible cardinality for their minimum distance.
In particular, they are maximal.
Therefore combining Proposition \ref{max} and \ref{murhod} we immediately 
obtain the following result.

\begin{corollary}\label{uppermrd}
 Let $\mC \subseteq \ma$ be an MRD code with $|\mC| \ge 2$. Then 
 $\rho(\mC) \le d(\mC) -1$. Moreover, equality holds if and only if 
 the maximality degree of $\mC$ is precisely $1$.
\end{corollary}

%
%
%
%

 The upper bound of Corollary \ref{uppermrd} is not sharp in general, as we show
in the following
example.
This proves in particular that not all MRD codes 
$\mC \varsubsetneq \F_q^{k \times m}$ with $|\mC| \ge 2$ can be nested into an MRD code 
$\mD \varsupsetneq \mC$ with
$d(\mD)=d(\mC)-1$.

\begin{example}
Take $q=2$ and $k=m=4$. Let $\mC$ be the linear code generated over $\F_2$ by the
following four matrices:
$$\begin{bmatrix}
   1 & 0 & 0 & 0 \\ 0 & 0 & 0 & 1 \\ 0 & 0 & 1 & 0 \\ 0 & 1 & 0 & 0
  \end{bmatrix}, \ \ \ \ \ 
\begin{bmatrix}
   0 & 1 & 0 & 0 \\ 0 & 0 & 1 & 1 \\ 0 & 0 & 0 & 1 \\ 1 & 1 & 0 & 0
  \end{bmatrix},  \ \ \ \ \ 
\begin{bmatrix}
   0 & 0 & 1 & 0 \\ 0 & 1 & 1 & 1 \\ 1 & 0 & 1 & 0 \\ 1 & 0 & 0 & 1
  \end{bmatrix},\ \ \ \ \ 
\begin{bmatrix}
   0 & 0 & 0 & 1 \\ 1 & 1 & 1 & 0 \\ 0 & 1 & 0 & 1 \\ 0 & 1 & 1 & 1
  \end{bmatrix}.\ \ \ \ \ 
$$  
We have $\dim(\mC)=4$ and $d(\mC)=4$. In particular, $\mC$ is a linear MRD codes.
On the other hand, one can check that $\rho(\mC)=2 \neq d(\mC)-1=3$, and that
$\mu(\mC)=2$.
\end{example}

We conclude observing that combining properties \ref{pr1}, \ref{pr2} and
\ref{pr4}
of Proposition \ref{generalpr} one can easily obtain the following general result on the 
puncturing of an MRD code.

\begin{corollary}
 Let $\mC \subseteq \ma$ be an MRD code. Then for any $A \in \mbox{GL}_k(\F_q)$
and for any $1 \le u \le k-1$ the punctured code 
$\Pi(\mC,A,u)$ is MRD as well.
\end{corollary}

Dually QMRD codes were proposed in \cite{duallypaper} as the best alternative 
to linear MRD codes for dimensions that are not multiples of $m$.
 A linear rank-metric code $\mC \subseteq \F_q^{k \times m}$
is \textbf{dually QMRD} if $\dim(\mC) \nmid m$ and the following two conditions hold:
$$d(\mC)=k-\lceil \dim(\mC)/m\rceil+1, \ \ \ \ \ \ \ \ 
d(\mC^\perp)=k-\lceil \dim(\mC^\perp)/m\rceil+1.$$

Clearly, a code is dually QMRD if and only if its dual code is dually QMRD.
The following proposition summarizes the most important properties of
dually QMRD codes.

\begin{lemma}[see Proposition 20 of \cite{duallypaper}]
\label{proprdqmrd}
 Let $\mC \subseteq \F_q^{k \times m}$ be a linear code. The following are
equivalent.
\begin{enumerate}
 \item $\mC$ is dually QMRD,
\item $\mC^\perp$ is dually QMRD,
\item $\dim(\mC) \nmid m$ and $d(\mC)+d(\mC^\perp)=k+1$.
\end{enumerate}
Moreover, the weight distribution of a dually
QMRD code $\mC$ is determined by $k$, $m$ and 
$\dim(\mC)$.
\end{lemma}

We now apply the external distance bound to derive an upper
bound 
on the covering radius
of dually QMRD codes. We start by computing the external distance, 
$\sigma^*(\mC)$, of a dually QMRD 
code $\mC$
of given parameters. Since 
$\mC$ is linear by definition, as in the proof 
of Corollary \ref{ei} we have 
$\sigma^*(\mC)= |\{i \in [k] : W_i(\mC^\perp) \neq 0\}|$.
We will need the following preliminary lemma.

\begin{lemma} \label{exist}
 Let $1 \le t \le km-1$ be any integer. There exist
linear codes $\mC \subsetneq \mD \subseteq \ma$ such that $\mC$ is dually QMRD, $\mD$ is MRD, 
$\dim(\mC)=t$ and $d(\mC)=d(\mD)$.
\end{lemma}

\begin{proof}
Let $\alpha:=\lfloor t/m \rfloor$. It is well known
(see e.g. the construction of \cite[Section 6]{D78} or \cite{john}) 
that there exist
linear MRD codes $\mE \subseteq \mD$ with
$\dim(\mE)=m \alpha$ and $\dim(\mD)=m(\alpha+1)$.
Let $\mE \varsubsetneq \mC \varsubsetneq \mD$ be a subspace with
$\dim(\mC)=t$. Since $\mE$ is MRD, it is maximal. Therefore $d(\mC)=d(\mD)$.
Now consider the nested codes $\mD^\perp \varsubsetneq \mC^\perp
\varsubsetneq \mE^\perp$. Since 
$\mD$ and $\mE$ are MRD, their dual codes $\mD^\perp$ and $\mE^\perp$ are
MRD as well (see \cite[Theorem 5.5]{D78} or \cite[Corollary 41]{alb} for a
simpler proof).
In particular, $\mD^\perp$ is maximal, and so $d(\mC^\perp)=d(\mE^\perp)$.
Since $\mD$ and $\mE^\perp$ are MRD, we have 
$d(\mD)= k-(\alpha+1)+1$ and $d(\mE^\perp)=k-(k-\alpha)+1$.
Therefore $$d(\mC)+d(\mC^\perp)=d(\mD)+d(\mE^\perp)=k-(\alpha+1)+1+k-(k-\alpha)+1=k+1,$$
and the result easily follows from Lemma \ref{proprdqmrd}. 
\end{proof}

We can now compute the external distance of a dually QMRD code.

\begin{theorem} \label{extdually}
 Let $\mC \subseteq \F_q^{k \times m}$ be a dually QMRD code.
Then $\sigma^*(\mC)=d(\mC)$. 
\end{theorem}

\begin{proof}
Since $\mC$ is linear, as in the proof  of Corollary 
\ref{ei} we have 
$\sigma^*(\mC)=|\{i \in [k] : W_i(\mC^\perp)>0\}|$.
By Lemma \ref{exist} there exist a dually QMRD code $\mC_1$ and a linear MRD code $\mD$
such that
$\mC_1 \subsetneq \mD$, $\dim(\mC)=\dim(\mC_1)$ and $d(\mC_1)=d(\mD)$.
Since $\mC$ and $\mC_1$ have the same dimension and are both dually QMRD, by 
Lemma \ref{proprdqmrd} the dual codes $\mC^\perp$ and $\mC_1^\perp$  have the
same weight distribution.
In particular, $\sigma^*(\mC)=\sigma^*(\mC_1)$.
Therefore it suffices to prove the theorem for the code $\mC_1$.
By Lemma \ref{proprdqmrd} we have
$d(\mC_1^\perp)=k+1-d(\mC_1)$. This clearly implies
\begin{equation} \label{ee2}
 \sigma^*(\mC_1) \le k-(k+1-d(\mC_1))+1 = d(\mC_1).
\end{equation}
On the other hand, by Corollary \ref{ei} we have 
$\sigma^*(\mC_1) \ge \rho(\mC_1)$, and by 
Lemma \ref{rd} we have $\rho(\mC_1) \ge d(\mD)$. Therefore
\begin{equation} \label{ee3}
 \sigma^*(\mC_1) \ge \rho(\mC_1) \ge d(\mD) = d(\mC_1).
\end{equation}
The theorem can now be easily obtained combining 
 inequalities (\ref{ee2}) and (\ref{ee3}).
\end{proof}

\begin{corollary} \label{upperdqmrd}
 The covering radius of a dually QMRD code
 $\mC$ satisfies
 $\rho(\mC) \le d(\mC)$. Moreover,
 equality holds if and only if 
 $\mC$ is not maximal.
\end{corollary}

\begin{proof}
 Combine Corollary \ref{ei}, Theorem \ref{extdually}, 
 Proposition \ref{murhod} and the fact that 
 $\mC$ is not maximal if and only if $\mu(\mC)=0$, by definition of 
 maximality degree.
\end{proof}

 The upper bound of Corollary \ref{upperdqmrd} is not sharp in general, as we
show in the following example.
This proves in particular that there exist dually QMRD codes
that are maximal. In particular, there exist dually QMRD codes
that are not contained into an MRD code with the same minimum distance.

\begin{example} \label{exdqmrd}
Take $q=2$ and $k=m=4$. Let $\mC$ be the linear code generated over $\F_2$ by the
following three matrices:

$$\begin{bmatrix}
   1 & 0 & 0 & 0 \\ 0 & 0 & 0 & 1 \\ 0 & 0 & 1 & 0 \\ 0 & 1 & 0 & 0
  \end{bmatrix}, \ \ \ \ \ 
\begin{bmatrix}
   0 & 1 & 0 & 0 \\ 1 & 0 & 1 & 1 \\ 0 & 0 & 0 & 1 \\ 1 & 1 & 0 & 0
  \end{bmatrix},  \ \ \ \ \ 
\begin{bmatrix}
   0 & 0 & 1 & 0 \\ 0 & 1 & 1 & 1 \\ 1 & 0 & 1 & 0 \\ 1 & 0 & 0 & 1
  \end{bmatrix}.
$$  
We have $\dim(\mC)=3$ and $d(\mC)=4$. Hence $\dim(\mC^\perp)=13$ and 
$d(\mC^\perp)=1$. Therefore $d(\mC)+d(\mC^\perp)=5$, and $\mC$ is dually
QMRD by Lemma \ref{proprdqmrd}. 
One can check that $\rho(\mC)=3 \neq d(\mC)=4$, and that $\mu(\mC)=1$.
\end{example}


\begin{thebibliography}{55}
	
 \bibitem{BGP15} D. Bartoli, M. Giulietti, I. Platoni, {\em On the Covering Radius of MDS Codes},
 IEEE Transactions on Information Theory, {\bf 61}, No. 2, 801--812, 2015. 	

\bibitem{C96} K. Chen, {\em On  the  Non-Existence  of  Perfect  Codes  with  Rank  Distance},
Mathematische  Nachrichten,  {\bf 182}, 89--98, 1996

 \bibitem{coveringcodes} G. Cohen, I. Honkala, S. Litsyn, A. Lobstein,
 \emph{Covering Codes}.
 North-Holland Mathematical Library, {\bf 54}, 1997.
 

  \bibitem{CKMS85} G. D. Cohen, M. G. Karpovsky, H. F. Mattson Jr., J. R.
Schatz, 
  \emph{Covering Radius -- Survey and Recent Results}. IEEE Transactions on
Information Theory, {\bf 31}, No. 3, 328--343, 1985.

\bibitem{CLLM97} G.D. Cohen, S.N. Litsyn, A.C. Lobstein, H.F. Mattson Jr.,
{\em Covering radius 1985-1994}, Applicable Algebra in Engineering, Communications and Computing, Vol. 8, No. 3, 173--239, 1997. 


\bibitem{duallypaper} J. de la Cruz, E. Gorla, 
H. Lopez, A. Ravagnani, \emph{Rank distribution of Delsarte codes}. Submitted,
online preprint: \url{https://arxiv.org/abs/1510.01008}.


  \bibitem{D73} P. Delsarte, \emph{Four Fundamental Parameters of a Code and
Their Combinatorial Significance}. Information and Control, {\bf 23},
407--438, 1973.

  \bibitem{D76} P. Delsarte, \emph{Association Schemes and $t$-Designs in
Regular Semilattices}. Journal of Combinatorial Theory, Series A, {\bf 20},
230--243, 1976.

  \bibitem{D78} P. Delsarte, \emph{Bilinear Forms over a Finite Field with
Applications to Coding Theory}. Journal of Combinatorial Theory, Series A,
{\bf 25}, 226--241, 1978.

 \bibitem{gabid} E. Gabidulin \emph{Theory of codes with maximum rank
distance}. 
Problems of Information Transmission, 1 (1985), 2, pp. 1 -- 12.


\bibitem{gadu1} M. Gadouleau, Z. Yan, \emph{Packing and
Covering Properties of Rank Metric Codes}. IEEE Transactions on Information
Theory {\bf 54}, No. 9, 3873--3883, 2008.


\bibitem{gadu2} M. Gadouleau, Z. Yan, \emph{Bounds on
Covering Codes with the Rank Metric}, IEEE Communications Letters, 
{\bf 13}, No. 9, 691--693, 2009.

\bibitem{GMR} V. Guruswami, D. Micciancio, O. Regev, 
{\em The complexity of the covering radius problem}, Computational Complexity, Vol. 14, No. 2, 90--121, 2005.

\bibitem{L06} P. Loidreau,	{\em A Welch-Berlekamp Like Algorithm for Decoding Gabidulin Codes}, Lect. Notes in Comp. Sc., pp. 36-45, 2006.

 \bibitem{McWS} F.J. MacWilliams, N.J.A. Sloane, \emph{The Theory of Error
Correcting Codes}. North-Holland Mathematical Library, {\bf 16}, 1978.

\bibitem{mesh} R. Meshulam, \emph{On the maximal rank in a subspace of
matrices}.  Quarterly Journal of Mathematics, 36 (1985), pp. 225 -- 229.


\bibitem{alblatt} A. Ravagnani, \emph{Duality of codes supported on regular
lattices, with an application to enumerative combinatorics}. Submitted,
online preprint: \url{https://arxiv.org/abs/1510.02383}.

\bibitem{alb} A. Ravagnani, \emph{Rank-metric codes and their duality theory}.
Designs, Codes and Cryptography, {\bf 80}, No. 1, 197--216, 2016.

\bibitem{john} J. Sheekey, \emph{A new family of linear maximum rank distance
codes}. Submitted, online preprint: \url{https://arxiv.org/abs/1504.01581}.


 \bibitem{ec} P. Stanley, \emph{Enumerative Combinatorics}, vol. 1, second
 edition. Cambridge University Press, 2012.
 
 \bibitem{W13} A. Wachter-Zeh, {\em Bounds on List Decoding of Rank-Metric Codes}, IEEE Trans. Inf. Theory, 59 (11) pp. 7268-7278, 2013.
 
 \bibitem{WAS13} A. Wachter-Zeh, V. Afanassiev, V. Sidorenko, {\em Fast decoding of Gabidulin Codes}, Designs, Codes and Cryptography, Vol. 66, No. 1, 57--73, 2013. 

\end{thebibliography}
\end{document}